\documentclass[12pt, reqno]{amsart}

\usepackage{upquote}

\usepackage{amsmath, amsthm, amscd, amsfonts, amssymb, graphicx, color}

\newtheorem{theorem}{Theorem}[section]

\newtheorem{corollary}[theorem]{Corollary}
\theoremstyle{definition}
\newtheorem{definition}[theorem]{Definition}

\theoremstyle{remark}
\newtheorem{remark}[theorem]{Remark}
\numberwithin{equation}{section}

\begin{document}
\setcounter{page}{1}

\title[Principal angles and approximation]
{Principal angles and approximation for quaternionic projections}

\author[T.~A.~Loring]{Terry A.~Loring$^1$}

\address{$^{1}$ Department of Mathematics and Statistics, University of New Mexico,
Albuquerque, NM 87131, USA.}
\email{\textcolor[rgb]{0.00,0.00,0.84}{loring@math.unm.edu}}

\dedicatory{This paper is dedicated to Professor Tsuyoshi Ando}



\date{June 2013 
\newline \indent $^{1}$ Corresponding author}

\begin{abstract}
We extend Jordan's notion of principal angles to work for two subspaces
of quaternionic space, and so have a method to analyze two orthogonal
projections in $\mathbf{M}_{n}(\mathbb{A})$ for $\mathbb{A}$ the
real, complex or quaternionic field (or skew field). From this we derive
an algorithm to turn almost commuting projections into commuting projections
that minimizes the sum of the displacements of the two projections.
We quickly prove what we need using the universal real C{*}-algebra
generated by two projections.
\end{abstract} \maketitle

\section{Two projections, the three-fold way}

The general form of two projections on complex Hilbert space is well
know, going back to at least Dixmier \cite{DixmierTwoSubspaces}.
The real case is older, being implicit in the work of
Jordan \cite{JordanTwoSubspaces}.
Restricted to the finite-dimensional case, we can think of these as
theorems about two projections in certain finite-dimensional real
$C^{*}$-algebras. One would therefore expect the same result to hold
in all finite-dimensional real $C^{*}$-algebras such as
$\mathbf{M}_{n}(\mathbb{H})$
where $\mathbb{H}$ is the skew field of quaternions.

Notation we will use gives us the needed supply of pairs
of small projections. For $0\leq\theta\leq\pi/2$ define
\[
P_{\theta}=\left[\begin{array}{cc}
1 & 0\\
0 & 0
\end{array}\right],\quad Q_{\theta}=\left[\begin{array}{cc}
\cos^{2}\theta & \cos\theta\sin\theta\\
\cos\theta\sin\theta & \sin^{2}\theta
\end{array}\right].
\]
By the dimension of a projection, we mean its trace, or the dimension
over $\mathbb{A}$ of its range.

\begin{theorem}
\label{thm:JordansLemma}
Suppose $\mathbb{A}$ equals $\mathbb{R}$,
$\mathbb{C}$ or $\mathbb{H}$. Suppose $P$ and $Q$ are projections
in $\mathbf{M}_{n}(\mathbb{A})$. If $\dim(P)\leq\dim Q$ then there
is a unitary $U$ in $\mathbf{M}_{n}(\mathbb{A})$ so that
\[
P=U
\left[\begin{array}{cccccc}
P_{\theta_{1}}\\
 & P_{\theta_{2}}\\
 &  & \ddots\\
 &  &  & P_{\theta_{J}}\\
 &  &  &  & 0I_{r}\\
 &  &  &  &  & 0I_{s}
\end{array}\right]
U^{*}\]
and
\[
Q
=
U
\left[\begin{array}{cccccc}
P_{\theta_{1}}\\
 & P_{\theta_{2}}\\
 &  & \ddots\\
 &  &  & P_{\theta_{J}}\\
 &  &  &  & I_{r}\\
 &  &  &  &  & 0I_{s}
\end{array}
\right]U^{*}
\]
where $r=\dim(P)-\dim(Q)$ and $s=n-r$. Moreover,
$\theta_{1},\dots,\theta_{J}$ and $r$ and $s$ are uniquely
determined, the $\theta_{j}$ up to order.
\end{theorem}

The real and complex cases are known. There is a short proof for the
quaternionic case involving universal real $C^{*}$-algebras.  We present
this proof in the next section. 

\begin{definition}
The \emph{principal angles} between the range of $P$ and the range
of $Q$ in $\mathbb{H}^{n}$ are $\theta_{1},\dots,\theta_{J}.$
\end{definition}

We can make sense of ``principle vectors'' if we consider a subspace
of $\mathbb{H}^{n}$ as a subspace of $\mathbb{C}^{2n}$ that is closed
under the anti-unitary symmetry
\[
\mathcal{T}\left(\left[\begin{array}{c}
\mathbf{v}\\
\mathbf{w}
\end{array}\right]\right)=\left[\begin{array}{c}
-\overline{\mathbf{w}}\\
\overline{\mathbf{v}}
\end{array}\right].
\]
In terms of Dyson's three-fold way \cite{DysonThreeFoldWay}, we are
discussing class AII. 
\begin{corollary}
\label{cor:KramersPairs} Suppose $M$ and $N$ are subspaces of $\mathbb{C}^{2n}$
with with $\mathcal{T}(M)=M$ and $\mathcal{T}(N)=N$. Then the principal
angles (see \cite{GalantaiPrincAngles}) have even multiplicity and
the principal vectors can be selected to be pairs of the form
$\mathbf{v},\mathcal{T}\mathbf{v}$.
\end{corollary}

\begin{proof}
If we replace these subspaces by $UM$and $UN$ for some unitary $U$
then $U$ will send principal vectors to principal vectors. If $U$
is symplectic then it will commute with $\mathcal{T}$
(see \cite{LoringQuaternions})
and so the conditions involving $\mathcal{T}$ will be preserved.
Thus we can assume, by Theorem~\ref{thm:JordansLemma}, that $M$
and $N$ are the ranges of $A$ and $B$ where 
\[
A=\left[\begin{array}{cc}
P\\
 & P
\end{array}\right],\quad Q=\left[\begin{array}{cc}
Q\\
 & Q
\end{array}\right]
\]
and
\[
P=\left[\begin{array}{cccccc}
P_{\theta_{1}}\\
 & P_{\theta_{2}}\\
 &  & \ddots\\
 &  &  & P_{\theta_{J}}\\
 &  &  &  & 0I_{r}\\
 &  &  &  &  & 0I_{s}
\end{array}\right],\quad Q=\left[\begin{array}{cccccc}
P_{\theta_{1}}\\
 & P_{\theta_{2}}\\
 &  & \ddots\\
 &  &  & P_{\theta_{J}}\\
 &  &  &  & I_{r}\\
 &  &  &  &  & 0I_{s}
\end{array}\right].
\]
Select the obvious real principal vectors for $P$ and $Q$ and then
double each such $\mathbf{v}$ as
\[
\left[\begin{array}{c}
\mathbf{v}\\
0
\end{array}\right],\left[\begin{array}{c}
0\\
\mathbf{v}
\end{array}\right]=\mathcal{T}\left[\begin{array}{c}
\mathbf{v}\\
0
\end{array}\right].
\]
\end{proof}

\begin{remark}
To unify things, we can regard a subspace of $\mathbb{R}^{n}$ as
a subspace of $\mathbb{C}^{n}$ that is closed under conjugation,
$\mathcal{T}_{+}(\mathbf{v})=\overline{\mathbf{v}}$. For two such
subspaces of $\mathbb{C}^{n}$ we can select principal vectors so
that each such $\mathbf{v}$ satisfies $\mathcal{T}_{+}(\mathbf{v})=\mathbf{v}$.
So in class AI we do not see the ``Kramers pairs'' effect that we
see in class AII, but in both cases principal vectors can be selected
to respect the relevant antiunitary symmetry.
\end{remark}

It is no doubt possible to prove Corollary~\ref{cor:KramersPairs}
directly, and then derive Theorem~\ref{thm:JordansLemma}. However,
the universal real $C^{*}$-algebra does much more than this. It can
be used to prove technical results relevant to real $K$-theory, or,
as we shall see, illuminate an algorithm for dealing with three
relatively easy classes of almost commuting matrices.

\section{A universal real $C^{*}$-algebra}

In this article we interpret ``real $C^{*}$-algebra'' to mean
specifically an $R^{*}$-algebra. A real Banach algebra $A$ with
involution is an $R^{*}$-algebra so long as its norm extends
to the complexification $A_{\mathbb{C}}$ to make that a
$C^{*}$-algebra. One can see \cite{sorensen2012SPandGraphs}
for a precise definition of what is allowed when doing relations on
$R^{*}$-algebras, but it certainly is allowed to say that a generator
$p$ satisfies $p^{2}=p^{*}=p$. For simplicity, we consider only
the case of $x_{1},\dots,x_{n}$ as generators. We say $\mathcal{U}$,
along with $\iota$ mapping $\{x_{1},\dots,x_{n}\}$ into $\mathcal{U}$,
is the universal $R^{*}$-algebra for a set of relations if the following
is true. Given any $R^{*}$-algebra $A$ with $y_{n},\dots,y_{n}$
satisfying those relations, there is a unique $*$-homomorphism
$\varphi:\mathcal{U}\rightarrow A$
so that $\varphi(\iota(x_{j}))=y_{j}$. Colloquially speaking, there
is always exactly one extension of the mapping $x_{j}\mapsto y_{j}$
to a $*$-homomorphism.

The following is very easy, given the machinery in developed by
S\o rensen in \cite{sorensen2012SPandGraphs}. We call it a Theorem only
because so much follows from it that is not so obvious.

\begin{theorem}
The universal $R^{*}$-algebra generated by two elements $p$ and
$q$ subject to the relations $p^{2}=p^{*}=p$ and $q^{2}=q^{*}=q$
is 
\[
\mathcal{B}=
\left\{
f\in C\left([0,\tfrac{\pi}{2}],\mathbf{M}_{2}(\mathbb{R})\right)
\left|
f(0)\in\left[\begin{array}{cc}
\mathbb{R} & 0\\
0 & 0
\end{array}\right]\mbox{ and }f(1)\in\left[\begin{array}{cc}
\mathbb{R} & 0\\
0 & \mathbb{R}
\end{array}\right]
\right.
\right\} 
\]
and the universal generators are $p_{0}$ and $q_{0}$ where 
\begin{equation}
p_{0}(t)=\left[\begin{array}{cc}
1 & 0\\
0 & 0
\end{array}\right],\quad q_{0}(t)=\left[\begin{array}{cc}
\cos^{2}(t) & \sin(t)\cos(t)\\
\sin(t)\cos(t) & \sin^{2}(t)
\end{array}\right].
\label{eq:p_0-and-q_0}
\end{equation}
\end{theorem}

\begin{proof}
The complexification of $\mathcal{B}$ is clearly 
\[
\mathcal{A}=\left\{ f\in C\left([0,1],\mathbf{M}_{2}(\mathbb{C})\right)\left|f(0)\in\left[\begin{array}{cc}
\mathbb{C} & 0\\
0 & 0
\end{array}\right]\mbox{ and }f(1)\in\left[\begin{array}{cc}
\mathbb{C} & 0\\
0 & \mathbb{C}
\end{array}\right]\right.\right\} 
\]
and this is known to be the universal complex $C^{*}$-algebra for
the relations of being two orthogonal projections. For example, see
\cite[\S 3]{PedersenMeasureII}. By \cite[Theorem 5.2.6.]{sorensen2012SPandGraphs},
the universal $R^{*}$-algebra for these relations is the closed real
$*$-algebra in $\mathcal{A}$ generated by $\{p_{0},q_{0}\}$, which
is $\mathcal{B}$.
\end{proof}

\subsection*{Proof of Theorem~\ref{thm:JordansLemma}}
Every finite-dimensional quotient of $\mathcal{B}$ is of the form
\[
\mathcal{C}=\mathbf{M}_{2}(\mathbb{R})\oplus\cdots\oplus\mathbf{M}_{2}(\mathbb{R})\oplus\mathbb{R}\oplus\cdots\oplus\mathbb{R}
\]
with any number of the $\mathbf{M}_{2}(\mathbb{R})$ and up to two
of the $\mathbb{R}$, with the surjection from $\mathcal{B}$ being
evaluation at various $t$ in $[0,1)$ and also 
\[
f\mapsto\left[\begin{array}{cc}
1 & 0\end{array}\right]f(1)\left[\begin{array}{c}
1\\
0
\end{array}\right]
\]
or
\[
f\mapsto\left[\begin{array}{cc}
0 & 1\end{array}\right]f(1)\left[\begin{array}{c}
0\\
1
\end{array}\right].
\]
The $*$-homomorphisms between finite-dimensional $R^{*}$-algebras
are known, say by \cite{GiordRealAF}. Up to unitary equivalence,
the only embedding of $\mathcal{C}$ into $\mathbf{M}_{n}(\mathbb{H})$
is found be the obvious embedding into
\[
\mathcal{D}
=
\mathbf{M}_{2}(\mathbb{H})
\oplus\cdots\oplus
\mathbf{M}_{2}(\mathbb{H})
\oplus
\mathbb{H}
\oplus\cdots
\oplus\mathbb{H}
\]
followed by an embedding that puts the $\mathbf{M}_{k}(\mathbb{H})$
down the diagonal, perhaps with multiplicity in each summand.

\subsection*{Computing principal vectors}
The standard for computing principal angles and vectors is an algorithm
by Bj\"orck and Golub \cite{BjorkGolubAngles}. Let us assume our
subspaces are given as the ranges of projections $P$ and $Q$. Their
algorithm first obtains partial isometries $E$ and $F$ so that $EE^{*}=P$
and $FF^{*}=Q$. Then a singular value decomposition $U\Omega V^{*}$
of $E^{*}F$ is computed, and the principal vectors are found by pairing
each column from $EU$ with a column from $FV$. 

We describe here a different algorithm. We have no particular application
in mind, so do not explore speed or accuracy issues. Moreover, the
algorithm is simpler if it is restricted to the case
$\left\Vert P-Q\right\Vert \leq1/\sqrt{2}$.
We use always the operator norm, so $\left\Vert X\right\Vert $ is
the largest singular value of $X$. See \cite{LoringQuaternions}
for details regarding the norm in the case of a matrix of quaternions. 

Following an idea from \cite{raeburnSincTwoProj}, we let $U$ be
the unitary in the polar decomposition of $X=QP+(I-Q)(I-P)$. We take
an orthonormal basis of eigenvectors for $PQP$, and for each $\mathbf{v}$
in that basis coming from an eigenspace at or above $\tfrac{1}{2}$
we find that $(\mathbf{v},U\mathbf{v})$ is a pair of principal vectors.
Assuming the eigen-decomposition is done with the appropriate symmetry
respected, the result will have the correct symmetry.

This algorithm can be validated, in exact arithmetic, from
Theorem~\ref{thm:JordansLemma}. Notice that the condition
$\left\Vert P-Q\right\Vert \leq1/\sqrt{2}$
causes the $\theta_{j}$to be at most $\tfrac{\pi}{4}$. For each
$P=P_{\theta}$ and $Q=Q_{\theta}$ we note that
\[
X=
\left[\begin{array}{cc}
\cos(\theta) & -\sin(\theta)\\
\sin(\theta) & \cos(\theta)
\end{array}\right]\left[\begin{array}{cc}
\cos(\theta) & 0\\
0 & \cos(\theta)
\end{array}\right]
\]
so
\[
U=\left[\begin{array}{cc}
\cos(\theta) & -\sin(\theta)\\
\sin(\theta) & \cos(\theta)
\end{array}\right]
\]
and the eigenvector for 
\[
PQP=\left[\begin{array}{cc}
\cos^{2}(\theta) & 0\\
0 & 0
\end{array}\right]
\]
for an eigenvalue above one-half will be 
\[
\left[\begin{array}{c}
1\\
0
\end{array}\right]
\]
and this will get paired with
\[
\left[\begin{array}{c}
\cos(\theta)\\
\sin(\theta)
\end{array}\right].
\]

Notice that $X$ will be invertible, and indeed have
$\left\Vert X\right\Vert \leq 1$ and 
$\left\Vert X^{-1}\right\Vert \leq\sqrt{2}$. Thus $U$ can
be quickly and accurately computed by Newton's method
\cite{HighamPolarDecomposition}.  Here is the algorithm
in Matlab, assuming that $p$ and $q$ are the projection matrices.
\begin{quotation}
\texttt{u = q{*}p + (eye(n)-q){*}(eye(n)-p); }

\texttt{for iteration = 1:5 }

\texttt{\quad{}u = (1/2){*}(u + inv(u')); }

\texttt{end }

\texttt{central = p{*}q{*}p; }

\texttt{central = 0.5{*}(central + central'); }

\texttt{{[}v,D{]} = eigs(central, dim); }

\texttt{a = v; }

\texttt{b = u{*}v; }
\end{quotation}
The pairs of principal vectors are in the columns of
\texttt{a} and \texttt{b}. Code that tests this is algorithm is
available as an auxiliary file to the arxiv.org preliminary
version of this paper.

\section{Almost commuting projections}

Almost commuting projections are much easier to understand than almost
commuting hermitian contractions. Indeed, Lin's
theorem \cite{LinAlmostCommutingHermitian}
is sufficiently difficult that there are no algorithms implementing
it. An algorithm for a related problem might be helpful. 

We can easily impose on our universal real $C^{*}$-algebra a
relation that bounds the commutator.
\begin{corollary}
Suppose $0\leq\delta<\tfrac{1}{2}$.
Let $C=\tfrac{1}{2}\arcsin(2\delta)$.
The universal $R^{*}$-algebra generated by two elements $p$ and
$q$ subject to the relations $p^{2}=p^{*}=p$ and $q^{2}=q^{*}=q$
and
\[
\left\Vert pq-qp\right\Vert \leq\delta
\]
is 
\[
\mathcal{B}_{\delta}
=
\left\{
f\in C\left(I_{C},\mathbf{M}_{2}(\mathbb{R})\right)
\left|
f(0)\in\left[\begin{array}{cc}
\mathbb{R} & 0\\
0 & 0
\end{array}\right]\mbox{ and }f(1)\in\left[\begin{array}{cc}
\mathbb{R} & 0\\
0 & \mathbb{R}
\end{array}\right]
\right.
\right\} 
\]
where
$I_{C} =[0,C]\cup[\tfrac{\pi}{2}-C,\tfrac{\pi}{2}]$ 
and the universal generators are $p_{0}$ and $q_{0}$ as
in equation~\ref{eq:p_0-and-q_0}.
\end{corollary}

If $P$ and $Q$ almost commute, and we have candidates $P'$ and
$Q'$ that are commuting projections, we can hope to have minimized
either
\[
\left\Vert P'-P\right\Vert +\left\Vert Q'-Q\right\Vert 
\]
or 
\[
\max\left(\left\Vert P'-P\right\Vert ,\left\Vert Q'-Q\right\Vert \right).
\]
In the first case, we can just let $P'=P$ and set $Q'$ to be the
spectral projection for $[\tfrac{1}{2},\infty)$ of
\[
PQP+(I-P)Q(I-P).
\]
This leads to the well-known result for the sum of displacements,
namely
\[
\left\Vert P'-P\right\Vert +\left\Vert Q'-Q\right\Vert 
=
\sin\left(\frac{1}{2}\arcsin(2x)\right).
\]

Controlling the max of the displacements does not seem to have been
considered before.

We observe that for $0\leq\theta\leq\pi/4$,
\[
\left\Vert P_{\theta}-Q_{\frac{\theta}{2}}\right\Vert 
=
\left\Vert Q_{\theta}-Q_{\frac{\theta}{2}}\right\Vert 
=
\sin\left(\frac{\theta}{2}\right)
\]
while for $\pi/4\leq\theta\leq\pi/2$, we let 
$\theta'=\frac{\theta}{2}+\frac{\pi}{4}$
and observe
\[
\left\Vert P_{\theta}-\left(I-Q_{\theta'}\right)\right\Vert 
=
\left\Vert Q_{\theta}-Q_{\theta'}\right\Vert 
=
\sin\left(\frac{\theta}{2}\right).
\]
For all $\theta$ we find
\[
\left\Vert P_{\theta}Q_{\theta}-Q_{\theta}P_{\theta}\right\Vert 
=
\frac{1}{2}\sin(2\theta).
\]
Finally, when we start with $0$ and $1$ or $0$ and $0$ we just
leave those alone.

\begin{theorem}
\label{thm:makeCommute}
Suppose $\mathbb{A}$ equals $\mathbb{R}$,
$\mathbb{C}$ or $\mathbb{H}$. If $P$ and $Q$ are projections in
$\mathbf{M}_{n}(\mathbb{A})$ then there are projections $P'$ and
$Q'$ in $\mathbf{M}_{n}(\mathbb{A})$ that commute and so that
\[
\left\Vert P-P'\right\Vert 
=
\left\Vert Q-Q'\right\Vert 
=
\sin\left(\frac{1}{4}\arcsin(2\delta)\right)
\]
where
\[
\delta=\left\Vert PQ-QP\right\Vert .
\]
The choice of $P^{'}$ and $Q^{'}$ can be made so that it is continuous
in $P$ and $Q$.
\end{theorem}

\begin{proof}
We can simply work in $\mathcal{B}_{\delta}$ and use the well-know
fact that naturality in $C^{*}$-algebra constructions leads to
continuity.
\end{proof}

\begin{theorem}
For $\delta=\left\Vert PQ-QP\right\Vert <\tfrac{1}{2},$ the commuting
projections $P'$ and $Q'$ of Theorem~\ref{thm:makeCommute} can
be computed by the following formulas: let
\begin{align*}
R & =\frac{1}{2}\left(PQP+QPQ\right)\\
S & =\frac{1}{2}\left(\left(I-P\right)Q\left(I-P\right)+Q\left(I-P\right)Q\right)\\
T & =PQP+\left(I-P\right)\left(I-Q\right)\left(I-P\right)
\end{align*}
and then let $E_{R}$ and $E_{S}$ be the spectral projections for
$R$ and $S$ corresponding to the set $\left[\tfrac{1}{5},\infty\right)$
and $E_{T}$ the spectral projections for $T$ corresponding to the
set $\left[\tfrac{1}{2},\infty\right),$ and finally 
\begin{align*}
P' & =E_{T}E_{R}E_{T}+(I-E_{T})(I-E_{s})(I-E_{T})\\
Q' & =E_{T}E_{R}E_{T}+(I-E_{T})E_{S}(I-E_{T})
\end{align*}
\end{theorem}

\begin{proof}
We notice that for
\[
p=\left[\begin{array}{cc}
1 & 0\\
0 & 0
\end{array}\right],\quad q=\left[\begin{array}{cc}
\cos^{2}(x) & \sin(x)\cos(x)\\
\sin(x)\cos(x) & \sin^{2}(x)
\end{array}\right],
\]
and with $0\leq x\leq\tfrac{\pi}{4}$, if we set
\[
r=\frac{1}{2}\left(pqp+qpq\right)
\]
then
\[
r=\left[\begin{array}{cc}
\cos\left(\frac{x}{2}\right) & \sin\left(\frac{x}{2}\right)\\
\sin\left(\frac{x}{2}\right) & -\cos\left(\frac{x}{2}\right)
\end{array}\right]\left[\begin{array}{cc}
\lambda_{1}(x) & 0\\
0 & \lambda_{2}(x)
\end{array}\right]\left[\begin{array}{cc}
\cos\left(\frac{x}{2}\right) & \sin\left(\frac{x}{2}\right)\\
\sin\left(\frac{x}{2}\right) & -\cos\left(\frac{x}{2}\right)
\end{array}\right]
\]
where
\[
\lambda_{1}(x)=\cos^{2}(x)\left(\frac{1}{2}\cos(x)+\frac{1}{2}\right)
\]
and 
\[
\lambda_{2}(x)=-\cos^{2}(x)\left(\frac{1}{2}\cos(x)-\frac{1}{2}\right).
\]

Suppose $e_{r}$ is the spectral projection of $r$
for $\left[\frac{1}{5},\infty\right)$.  Since
\[
\lambda_{2}(x)\leq\lambda_{2}(\tfrac{\pi}{4})
=
\frac{2-\sqrt{2}}{8}\leq\frac{1}{5}
\leq
\frac{2+\sqrt{2}}{8}
\leq
\lambda_{2}(\tfrac{\pi}{4})
\leq
\lambda_{2}(x)
\]
we find that 
\begin{figure}
\includegraphics[bb=81bp 240bp 529bp 544bp,clip,scale=0.65]{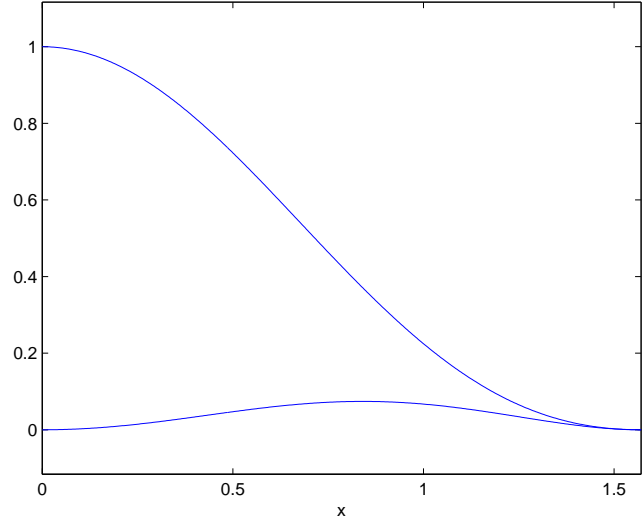}
\caption{The two eigenvalues of $r(x)$ for various scalar values of $x$. 
\label{fig:The-two-eigenvalues}}
\end{figure}
\[
e_{r}=\left[
\begin{array}{cc}
\cos^{2}\left(\frac{x}{2}\right) & \sin\left(\frac{x}{2}\right)\cos\left(\frac{x}{2}\right)\\
\sin\left(\frac{x}{2}\right)\cos\left(\frac{x}{2}\right) & \sin^{2}\left(\frac{x}{2}\right)
\end{array}
\right]
\]
which is on the midpoint of the canonical path between $p$ and $q$.
(See \cite{BrownMetricSubspaces}.) By symmetry, set
\[
s=\frac{1}{2}\left((1-p)q(1-p)+q(1-p)q\right)
\]
and find that the spectral projection $e_{s}$ 
of $s$ for $\left[\frac{1}{5},\infty\right)$ satisfies
\[
e_{s}=\left[
\begin{array}{cc}
\cos^{2}\left(\frac{\pi}{2}-\frac{x}{2}\right) & \sin\left(\frac{x}{2}\right)\cos\left(\frac{x}{2}\right)\\
\sin\left(\frac{x}{2}\right)\cos\left(\frac{x}{2}\right) & \sin^{2}\left(\frac{x}{2}\right)
\end{array}
\right],
\]
and so for $\frac{\pi}{4}\leq x\leq\frac{\pi}{2}$, and this is the
``midpoint'' between $1-p$ and $q$. These projections do not become
zero when $x$ is in the opposite subinterval, as 
indicated by Figure~\ref{fig:The-two-eigenvalues}.

\begin{figure}
\includegraphics[bb=96bp 240bp 517bp 544bp,clip,scale=0.65]{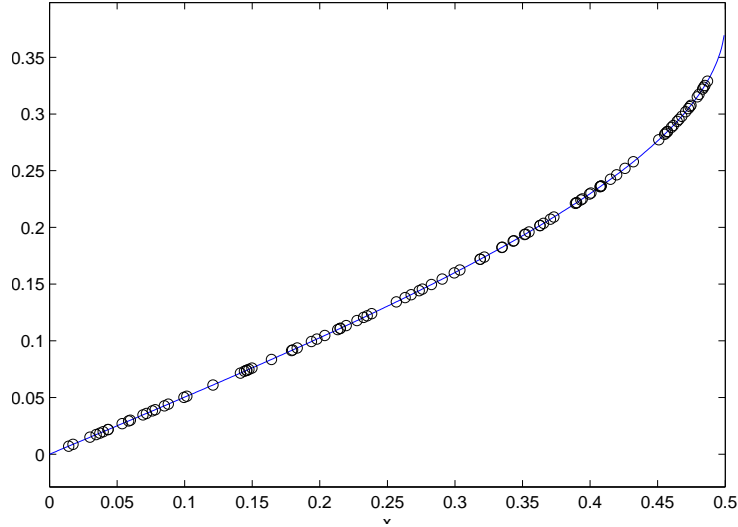}
\caption{Distance to computed commuting projections by the Formulas in 
implemented in Matlab. There were 100 pairs of 200-by-200 real
projections of distance at most 0.49 apart. The blue line is
the exact answer of
$\sin(\arcsin(2x)/4)$.
\label{fig:Distance-to-computed}
}
\end{figure}

For all $x$ we use
\[
t=pqp+(1-p)(1-q)(1-p)
\]
which is
\[
t=\left[\begin{array}{cc}
\cos^{2}(x) & 0\\
0 & \sin^{2}(x)
\end{array}\right].
\]
Thus the spectral projection $e_{t}$ for $t$ 
corresponding to $ $$\left[\frac{1}{2},\infty\right)$
is 
\[
\left[\begin{array}{cc}
1 & 0\\
0 & 0
\end{array}\right]
\]
for $x$ less than $\frac{\pi}{2}$ and 
\[
\left[\begin{array}{cc}
0 & 0\\
0 & 1
\end{array}\right]
\]
for $x$ greater than $\frac{\pi}{2}$. 

\begin{figure}
(a)\includegraphics[bb=96bp 240bp 517bp 564bp,clip,scale=0.65]{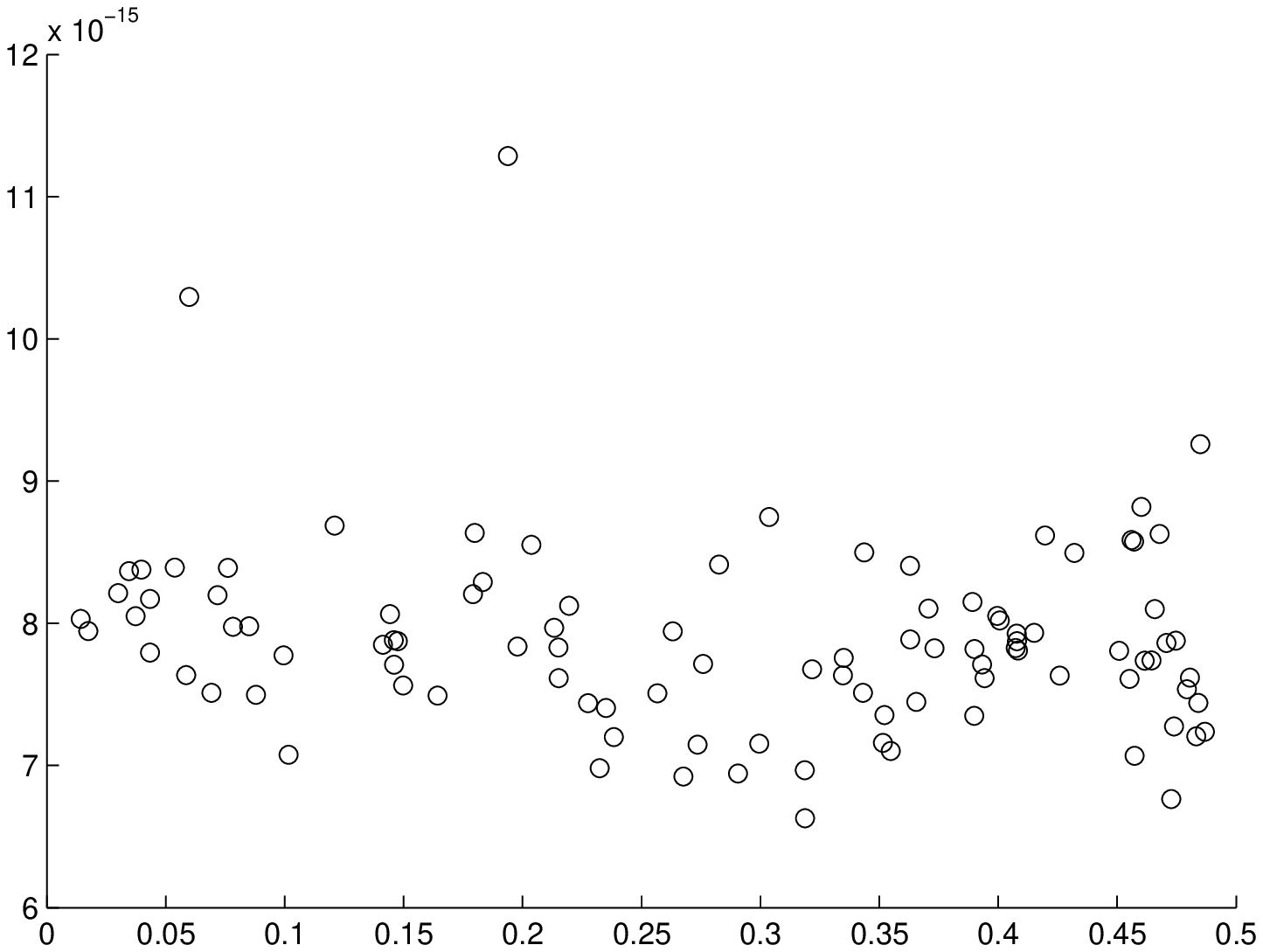}

(b)\includegraphics[bb=96bp 240bp 517bp 564bp,clip,scale=0.65]{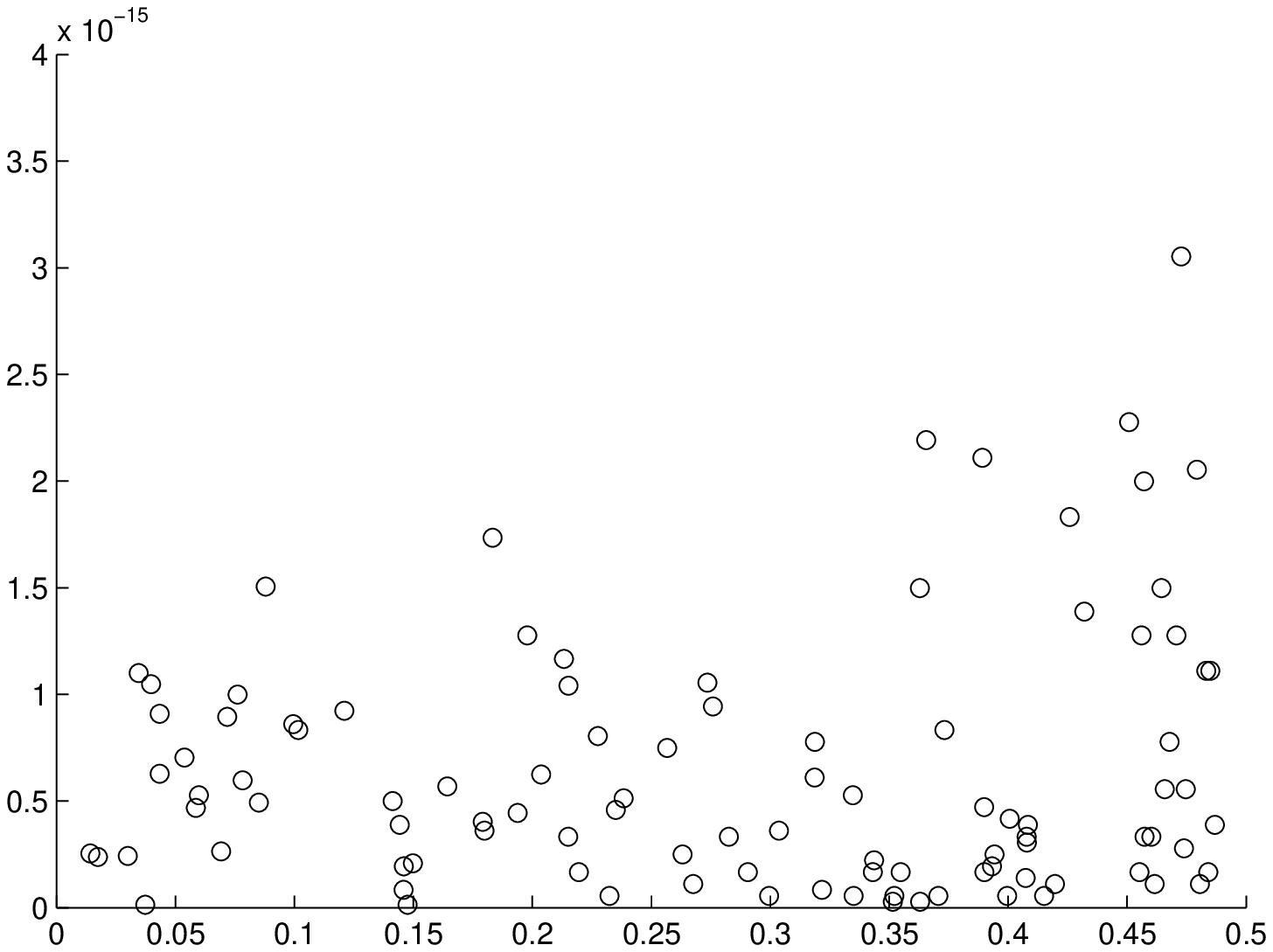}

\caption{The errors with the same test matrices as in 
Figure~\ref{fig:Distance-to-computed}.
(a) The sum of errors, in operator norm, in the relations
$P^{'2}=P^{'}$,$Q^{'2}=Q^{'}$,$P^{'*}=P^{'}$,$Q^{'*}=Q^{'}$ and
$P^{'}Q^{'}=Q^{'}P^{'}.$ (b) The errors from the optimal
in 
$\max\left(\left\Vert P^{'}-P\right\Vert ,\left\Vert Q^{'}-Q\right\Vert \right)$.
\label{fig:Errors-in-output.}
}
\end{figure}

If we define $r$ and $s$ and $t$ by the above formulas and set
\begin{align*}
p_{1} & =e_{t}e_{r}e_{t}+(1-e_{t})(1-e_{s})(1-e_{t})\\
q_{1} & =e_{t}e_{r}e_{t}+(1-e_{t})s(1-e_{t})
\end{align*}
and find these are exactly commuting projections and
\[
\max
\left(
\left
\Vert p-p_{1}\right\Vert,
\left\Vert q-q_{1}\right\Vert \right)=\sin\left(\frac{x}{2}
\right).
\]
Since
\[
\delta=\left\Vert pq-qp\right\Vert =\sin(x)\cos(x)=\frac{1}{2}\sin(2x)
\]
we have
\[
\max
\left(
\left\Vert p-p_{1}\right\Vert ,\left\Vert q-q_{1}\right\Vert 
\right)
=
\sin\left(\frac{1}{4}\arcsin(2\delta)\right).
\]
\end{proof}

Some applications of the half-angle formula give us
\[
\sin\left(\frac{1}{4}\arcsin(2\delta)\right)
=
\sqrt{\frac{1}{2}-\frac{1}{2}\sqrt{\frac{1}{2}+\frac{1}{2}\sqrt{1-4\delta^{2}}}}
\]
should someone think this is an improvement.

This is readily programmable for complex matrices, and can be done
so real and quaternionic matrices lead to real and quaternionic matrices
during the calculation. Code that tests this is available as an auxiliary
file to the arxiv.org preliminary version of this paper. The results
on real matrices is shown in Figure~\ref{fig:Distance-to-computed}
with numerical errors shown in Figure~\ref{fig:Errors-in-output.}.
The data as shown were created with \texttt{testCommute(200,100)} using
the code in the auxiliary file \texttt{testCommute.m}.

{\bf Acknowledgement.} This work was partially
supported by a grant from the Simons Foundation (208723 to Loring).

\bibliographystyle{amsplain}

\begin{thebibliography}{10}

\bibitem{BjorkGolubAngles}
{\.A}ke Bj{\"o}rck and Gene~H. Golub, \emph{Numerical methods for computing
  angles between linear subspaces}, Math. Comp. \textbf{27} (1973), 579--594.


\bibitem{BrownMetricSubspaces}
Lawrence~G. Brown, \emph{The rectifiable metric on the set of closed subspaces
  of {H}ilbert space}, Trans. Amer. Math. Soc. \textbf{337} (1993), no.~1,
  279--289. 

\bibitem{DixmierTwoSubspaces}
Jacques Dixmier, \emph{Position relative de deux vari\'et\'es lin\'eaires
  ferm\'ees dans un espace de {H}ilbert}, Revue Sci. \textbf{86} (1948),
  387--399. 

\bibitem{DysonThreeFoldWay}
Freeman~J. Dyson, \emph{Statistical theory of the energy levels of complex
  systems. {I}}, J. Mathematical Phys. \textbf{3} (1962), 140--156.

\bibitem{GalantaiPrincAngles}
A.~Gal{\'a}ntai and Cs.~J. Heged{\H{u}}s, \emph{Jordan's principal angles in
  complex vector spaces}, Numer. Linear Algebra Appl. \textbf{13} (2006),
  no.~7, 589--598. 

\bibitem{GiordRealAF}
T.~Giordano, \emph{A classification of approximately finite real
  {$C^*$}-algebras}, J. Reine Angew. Math. \textbf{385} (1988), 161--194.
  

\bibitem{HighamPolarDecomposition}
Nicholas~J. Higham, \emph{Computing the polar decomposition---with
  applications}, SIAM J. Sci. Statist. Comput. \textbf{7} (1986), no.~4,
  1160--1174. 

\bibitem{JordanTwoSubspaces}
Camille Jordan, \emph{Essai sur la g\'eom\'etrie \`a $n$ dimensions}, Bull.
  Soc. Math. France \textbf{3}, 103--174.

\bibitem{LinAlmostCommutingHermitian}
Huaxin Lin, \emph{Almost commuting selfadjoint matrices and applications},
  Operator algebras and their applications ({W}aterloo, {ON}, 1994/1995),
  Fields Inst. Commun., vol.~13, Amer. Math. Soc., Providence, RI, 1997,
  pp.~193--233. 

\bibitem{LoringQuaternions}
Terry~A. Loring, \emph{Factorization of matrices of quaternions}, Exposition.
  Math. \textbf{30} (2012), no.~3, 250--267.

\bibitem{PedersenMeasureII}
Gert~Kjaerg{\.a}rd Pedersen, \emph{Measure theory for {$C^{\ast}$}-algebras.
  {II}}, Math. Scand. \textbf{22} (1968), 63--74. 

\bibitem{raeburnSincTwoProj}
Iain Raeburn and Allan~M Sinclair, \emph{The {$C^{\ast}$}-algebra generated by two
  projections}, School of Mathematics, University of New South Wales, 1989.

\bibitem{sorensen2012SPandGraphs}
A.P.W. S{\o}rensen, \emph{Semiprojectivity and the geometry of graphs}, Ph.D.
  thesis, University of Copenhagen, 2012, {w}ww.math.ku.dk/noter/filer.


\end{thebibliography}

\end{document}